\documentclass[pdflatex,sn-mathphys-num]{sn-jnl}
\usepackage{graphicx}
\usepackage{amsmath,amssymb,amsfonts}
\usepackage{amsthm}
\usepackage{pgfplots}
\pgfplotsset{compat=newest}

\theoremstyle{thmstyleone}
\newtheorem{theorem}{Theorem}[section]
\newtheorem{proposition}[theorem]{Proposition}
\newtheorem{lemma}[theorem]{Lemma}

\theoremstyle{thmstyletwo}

\newtheorem{remark}{Remark}
\newtheorem{assumption}{Assumption}

\theoremstyle{thmstylethree}

\newcommand{\rme}{\mathrm{e}}
\newcommand{\rmi}{\mathrm{i}}
\newcommand{\RR}{\mathbb{R}}
\newcommand{\CC}{\mathbb{C}}

\begin{document}

\title[Analysis of split exponential
  integrators for semilinear problems
  with noncommuting operators]{Analysis of split exponential
  integrators for semilinear problems
  with noncommuting operators}

\author[1]{\fnm{Luca} \sur{Battestini}}\email{luca.battestini@unitn.it}

\author[2]{\fnm{Marco} \sur{Caliari}}\email{marco.caliari@univr.it}

\author[3]{\fnm{Fabio} \sur{Cassini}}\email{fabio.cassini@unipd.it}

\affil[1]{\orgdiv{Department of Mathematics}, \orgname{University of Trento},
\country{Italy}}

\affil[2]{\orgdiv{Department of Computer Science},
  \orgname{University of Verona},
\country{Italy}}

\affil[3]{\orgdiv{Department of Mathematics ``Tullio Levi-Civita''},
  \orgname{University of Padua}, \country{Italy}}

\abstract{Splitting the exponential and
  exponential-like functions can significantly reduce
  the overall computational cost of exponential integrators,
  providing a strong motivation for their rigorous analysis.
  In this paper, we analyze splitting errors for two noncommuting
  unbounded operators and derive new 
  representations that include the popular Lie--Trotter and Strang cases.
  Moreover, within the abstract framework of
  strongly continuous semigroups, we prove the convergence of a split
  version of the exponential Euler method and a second-order exponential
  Runge--Kutta method for semilinear problems.
  Numerical experiments on a nonlinear dispersive
  equation and   a diffusion--reaction equation demonstrate the theoretical findings,
  stability properties, and efficiency of the proposed
  split exponential integrators compared
  to the state of the art.}


\maketitle

\section{Introduction}\label{sec:introduction}

We consider the abstract evolution problem
\begin{equation}\label{eq:pde}
    \begin{cases}
        \begin{aligned}
            u'(t) &= Lu(t) + g(u(t))=(A+B)u(t) + g(u(t)),\quad t\in(0,T], \\
            u(0) &= u_0,
        \end{aligned}
    \end{cases} 
\end{equation} 
where $u$ is the unknown, $A$ and $B$ are  linear noncommuting
unbounded operators, and $g$ is the
nonlinear term.
The precise assumptions on the operators, nonlinearity, and solution will
be given in the following section.
Typically, 
the analytical solution of~\eqref{eq:pde} cannot be obtained in explicit
form, and therefore it is important to develop numerical
methods for computing an approximate one. In many applications, 
problems in the form~\eqref{eq:pde} arise from the spatial discretization of
semilinear Partial Differential Equations (PDEs), like
advection--diffusion--reaction equations or nonlinear Schr\"odinger
equations. In these cases, the operators are matrices of arbitrary
large norm representing the discretization of differential
operators with their boundary conditions.
The resulting system of Ordinary Differential Equations (ODEs)
is therefore \emph{stiff}, and standard explicit time-marching
schemes require small time steps to work properly.
For these problems, where the linear part is responsible of the stiffness,
exponential integrators of Runge--Kutta type
(see \cite{HO10} for a seminal review)
have proved successful,
since they are explicit and don't suffer of step size restrictions.
However,
they require exponential-like matrix functions, denoted by
$\varphi_\ell$, evaluated at scales of $L=A+B$, whose computation can be an expensive task.
A promising
way to reduce the computational burden is by employing suitable
splitting approximations of $\varphi_\ell(\tau L)$, being $\tau$
the time step size, involving  $\varphi_\ell(\tau A)$ and
$\varphi_\ell(\tau B)$ separately. Indeed, this approach can be more efficient
in many situations,
e.g., when leveraging matrix- or tensor-based approaches that exploit
the spatial structure of a multidimensional problem. For applications to directional splitting
in hyperrectangular spatial domains we invite the reader
to consult~\cite{C24,CC24,CC24bis},
and \cite{CC27} for a recent extension to curvilinear domains.
The convergence analysis of
the resulting split exponential integrators in the context
of parabolic problems with smooth solutions
has been recently studied in the case of \emph{commuting}
operators in \cite{CCEO26}, and a generalization to nonsmooth solutions has
been presented in~\cite{HOZ26}.

In this paper, we are interested in the convergence analysis
of two split exponential integrators of order one and two,
namely ERK1L and ERK2L, that are based on the aforementioned splitting
approximations.
To this aim, we derive new error formulas
for the splitting of $\varphi_\ell$ functions
in the case of \emph{noncommuting} operators, including
the well-known Lie--Trotter and Strang splittings for the exponential
function ($\ell=0$).
More in detail, in Section~\ref{sec:methods}, we introduce new elementary
representations of the splitting errors in the abstract framework of
strongly continuous semigroups. In addition, we analyze the proposed
split versions of the first-order exponential Euler method ETD1RK and of
the second-order exponential Runge--Kutta method ETD2RK.
In Section~\ref{sec:examples}, we present some numerical examples, where
we show the effectiveness of the split integrators,
thanks to the employment of the fast Fourier transform and a matrix-oriented
approach. Also, we highlight
their favourable stability properties, and compare their
performances with unsplit methods. Finally, in
Section~\ref{sec:conclusions} we draw the conclusions and discuss possible
future developments.

\section{Splitting approximations and
  exponential integrators}\label{sec:methods}

In this section, we present two split exponential integrators for the solution
of \eqref{eq:pde}. For each scheme, we analyze the employed
splitting approximations and
present, under suitable assumptions, the convergence analysis.
All the results are obtained in the abstract framework of semigroups. In the
following, the time interval $[0,T]$ is
uniformly discretized by the time steps $t_n=n\tau$, where $\tau$
is the constant time step size. We start by stating the following
assumption on the operators $A$, $B$, and $L=A+B$.

\begin{assumption}\label{assumption:semigroup}
Let $(X,\lVert\cdot\rVert)$ be a Banach space. The
operators $A$, $B$, and $L=A+B$, defined on their respective domains
and taking values in $X$, generate strongly continuous semigroups
$\{\rme^{tA}\}_{t\geq 0}$,
$\{\rme^{tB}\}_{t\geq 0}$, and $\{\rme^{tL}\}_{t\geq 0}$, respectively.
\end{assumption} 
Consequently, the following bounds hold true:
\begin{equation}\label{eq:sgbound}
    \lVert\rme^{tA}\rVert \leq C,\quad \lVert\rme^{tB}\rVert \leq C,\quad
    \lVert\rme^{tL}\rVert \leq C,    \quad 0\leq t \leq T.
\end{equation}
Throughout the paper, we denote by $C$ a general constant that is independent
of the stiffness of the problem. From the bounds in \eqref{eq:sgbound},
it follows that
\begin{equation}\label{eq:phibound}
  \lVert\varphi_\ell(tA)\rVert\leq C,\quad
  \lVert\varphi_\ell(tB)\rVert\leq C,\quad
  \lVert\varphi_\ell(tL)\rVert\leq C,
  \quad 0\leq t \leq T,\quad \ell\ge0,
\end{equation}
where the $\varphi_\ell$ functions are defined by
\begin{equation}\label{eq:phifun}
    \varphi_0(z)=\rme^z,\quad \varphi_k(z)
    =\frac{1}{(k-1)!}\int_0^1 \theta^{k-1}\rme^{(1-\theta)z}d\theta 
    = \sum_{i=0}^{\infty}\frac{z^i}{(k+i)!},\quad k>0,
\end{equation}
and satisfy the recurrence relation
\begin{equation}\label{eq:phirel}
    \varphi_{\ell}(z) -\frac{1}{\ell!}=\varphi_{\ell+1}(z)z,\quad \ell\geq 0.
\end{equation}
We emphasize that we are interested in the case
of two noncommuting operators, hence
\begin{equation*}
    \rme^{tL}\neq\rme^{tA}\rme^{tB},
\end{equation*}
and, obviously,
\begin{equation*}
    \varphi_k(tL)\neq\varphi_k(tA)\varphi_k(tB), \quad k >0.
\end{equation*}
\subsection{Split exponential integrator ERK1L}\label{ssec:serk1}
We start by considering the exponential Euler method \cite[formula (1.6)]{HO10}
\begin{equation}\label{eq:eeg}
    \tilde u_{n+1} = \rme^{\tau L}\tilde u_n+\tau\varphi_1(\tau L)g(\tilde u_n),
\end{equation}
also known as ETD1RK.
As mentioned before, we want to approximate the exponential and the
$\varphi_1$ function evaluated at $\tau L$, deriving a first-order
scheme that can be
efficiently implemented. To this aim, we consider the approximation
\begin{equation}\label{eq:phisplit}
    \varphi_\ell(\tau L)\approx\ell!\varphi_\ell(\tau A)\varphi_\ell(\tau
    B),\quad \ell\geq 0,
\end{equation}
first introduced in \cite{CC24} and analyzed in the \emph{commuting} case
in \cite{CCEO26}. 
In the case of bounded operators,
it is straightforward to observe that the error
of this approximation is proportional to $\tau^2$
by exploiting the Taylor expansion of the $\varphi_\ell$ function.
Therefore, under suitable assumptions, we expect that
 \eqref{eq:phisplit} could be successfully employed
in our stiff framework. In fact, we propose
the split exponential Euler method 
\begin{equation}\label{eq:erk1l}
    u_{n+1} = \rme^{\tau A}\rme^{\tau B}u_n+\tau\varphi_1(\tau A)\varphi_1(\tau B)
    g(u_n),
\end{equation}
which we call ERK1L. Notice that, differently from the commuting case,
changing the order of the operators in the formula above would
lead to different schemes.

To prove the convergence of the scheme, we first analyze
splitting \eqref{eq:phisplit}, deriving an exact expression for the
error.
\begin{lemma}\label{lemma:phisplit}
  Let $\ell\ge 0$.
  The following representations of the
  splitting error for \eqref{eq:phisplit} hold:
    \begin{subequations}\label{eq:errsplit}
        \begin{align}
            &\varphi_\ell(\tau L)-\ell !\varphi_\ell(\tau A)\varphi_\ell(\tau B)
            \nonumber \\  
            &=\tau\left(\varphi_{\ell+1}(\tau L)L-\varphi_{\ell+1}(\tau A)A
            -\ell !\varphi_\ell(\tau A)\varphi_{\ell+1}(\tau B)B\right)
            \label{eq:errsplit1}\\
            &=\tau^2\left(\varphi_{\ell+2}(\tau L)L^2
            -\varphi_{\ell+2}(\tau A)A^2
            -\ell !\varphi_\ell(\tau A)\varphi_{\ell+2}(\tau B)B^2
            -\frac{1}{\ell+1}\varphi_{\ell+1}(\tau A)AB\right). 
            \label{eq:errsplit2}
        \end{align}
    \end{subequations}
\end{lemma}
\begin{proof}
  The idea of the proof is to repeatedly employ relation~\eqref{eq:phirel},
  adding and subtracting suitable quantities. In particular, we have
\begin{equation*}
    \begin{aligned}
        \varphi_\ell(\tau L)-\ell !\varphi_\ell(\tau A)\varphi_\ell(\tau B)
        &= \varphi_\ell(\tau L)-\frac{1}{\ell!}I
        -\left(\varphi_\ell(\tau A)-\frac{1}{\ell!}I\right)
        -\ell!\varphi_\ell(\tau A)\left(\varphi_\ell(\tau B)
        -\frac{1}{\ell!}I\right) \\
        &= \tau\left(\varphi_{\ell+1}(\tau L)L-\varphi_{\ell+1}(\tau A)A
        -\ell !\varphi_\ell(\tau A)\varphi_{\ell+1}(\tau B)B\right),
    \end{aligned}
\end{equation*}
which is in fact~\eqref{eq:errsplit1}. Then, proceeding in the same way,
we write
\begin{equation*}
    \begin{aligned}
        &\tau\left(\varphi_{\ell+1}(\tau L)L-\varphi_{\ell+1}(\tau A)A
        -\ell !\varphi_\ell(\tau A)\varphi_{\ell+1}(\tau B)B\right) \\
        &= \tau \Bigg( \left(\varphi_{\ell+1}(\tau L)
        -\frac{1}{(\ell+1)!}I\right)L-\left(\varphi_{\ell+1}(\tau A)
        -\frac{1}{(\ell+1)!}I\right)A \\
        &\quad-\ell!\varphi_\ell(\tau A)\left(\varphi_{\ell+1}(\tau B)
        -\frac{1}{(\ell+1)!}I\right)B-\frac{1}{\ell+1}\left(\varphi_\ell(\tau A)
        -\frac{1}{\ell!}I\right)B\Bigg)\\
        &= \tau^2\left(\varphi_{\ell+2}(\tau L)L^2
        -\varphi_{\ell+2}(\tau A)A^2
        -\ell !\varphi_\ell(\tau A)\varphi_{\ell+2}(\tau B)B^2
        -\frac{1}{\ell+1}\varphi_{\ell+1}(\tau A)AB\right),
    \end{aligned}
\end{equation*}
thus obtaining formula~\eqref{eq:errsplit2}.
\end{proof}
Notice that in formula~\eqref{eq:errsplit2} the total degree of the monomials
made by the unbounded operators
$A$ and $B$ is two.
Also, for $\ell=0$ the result is a representation of the Lie--Trotter splitting
error for the exponential function, which
differs from those presented in~\cite{HO09,IK24} since it shows the products
of $A$ and $B$ in the rightmost positions.

With the aid of the lemma, we now derive the local error bounds that we
will use in the convergence proof of ERK1L. For brevity of notation,
we define the sets
\begin{equation*}
    \mathcal{D}_1 = \mathcal{D}(A)\cap\mathcal{D}(B),\quad
  \mathcal{D}_2 = \mathcal{D}(A^2)\cap\mathcal{D}(B^2)\cap\mathcal{D}(AB)
    \cap\mathcal{D}(BA).
\end{equation*}
\begin{proposition}\label{prop:splitorder}
  Under Assumption~\ref{assumption:semigroup}, the following bounds
  on the local error of approximation~\eqref{eq:phisplit} hold true:
  \begin{enumerate}
    \item if $v\in\mathcal{D}_1$, then
  \begin{equation*}
        \lVert\varphi_\ell(\tau L)v-\ell!\varphi_\ell(\tau A)
        \varphi_\ell(\tau B)v\rVert \leq C\tau;
  \end{equation*}
    \item if $v\in\mathcal{D}_2$, then
  \begin{equation*}
        \lVert\varphi_\ell(\tau L)v-\ell!\varphi_\ell(\tau A)
        \varphi_\ell(\tau B)v\rVert \leq C\tau^2.
  \end{equation*}
  \end{enumerate}
\end{proposition}
\begin{proof}
  The results can be proved combining bounds~\eqref{eq:phibound} with
   identities \eqref{eq:errsplit1} and
   \eqref{eq:errsplit2}, and taking into account that
   $v\in\mathcal{D}_1$ and $v\in\mathcal{D}_2$, respectively.
\end{proof}
To present the convergence result for the scheme ERK1L,
we need two additional assumptions.
\begin{assumption}\label{assumption:contraction}
The semigroups generated by the operators $A$ and $B$ satisfy
the following bounds:
\begin{equation*}
  \lVert\rme^{t A}\rVert\le 1,
  \quad\lVert\rme^{t B}\rVert\leq 1,\quad 
    0\leq t \leq T.
\end{equation*}
\end{assumption}
Notice that, if the previous contraction bounds
are not directly satisfied, in some cases a standard shift and renorming
procedure can be applied, for instance when $A$ generates a strongly
continuous semigroup and $B$ is a bounded operator.
\begin{assumption}\label{assumption:order1}
  Problem~\eqref{eq:pde} has solution $u(t)\in\mathcal{D}_2$, and
$g (u(t))\in\mathcal{D}_1$. Moreover, $g$ is Fr{\'e}chet
differentiable, and its derivative is uniformly bounded.
\end{assumption}
We are now in a position to state the main result of the subsection.

\begin{theorem}\label{thm:serk1}
Let Assumptions \ref{assumption:semigroup}, \ref{assumption:contraction} and 
\ref{assumption:order1} be valid. Then, the following bound on the global
error 
of scheme ERK1L~\eqref{eq:erk1l} holds:
\begin{equation*}
    \lVert u_n-u(t_n)\rVert\leq C\tau,\quad 0\leq n\tau\leq T.
\end{equation*}
The constant $C$ is independent of $\tau$ and $n$, but it depends on $T$.
\end{theorem}
\begin{proof}
The exact solution of \eqref{eq:pde} can be written by means of the 
variation-of-constants formula as
\begin{equation}\label{eq:voc}
    u(t_{n+1}) = \rme^{\tau L}u(t_n)+\int_0^\tau \rme^{(\tau-s)L}
    g(u(t_n+s))ds.
\end{equation}
Now, we define
\begin{equation*}
  h(t)=g(u(t))
  \end{equation*}
and consider its Taylor expansion
\begin{equation}\label{eq:h}
    h(t_n+s) = h(t_n)+\int_0^s h'(t_n+\sigma)d\sigma.
\end{equation}
Then, we take the difference between the approximate solution and 
\eqref{eq:voc}. Inserting expansion~\eqref{eq:h}, we obtain  
\begin{equation*} 
\begin{aligned}
    u_{n+1}-u(t_{n+1}) 
    &= \rme^{\tau A}\rme^{\tau B}u_n-\rme^{\tau L}u(t_n)+\tau \varphi_1(\tau
    A)\varphi_1(\tau B)g(u_n)-\tau\varphi_1(\tau L)h(t_n)
    \\&\quad-\int_0^\tau \rme^{(\tau-s)L}\int_0^s h'(t_n+\sigma)d\sigma ds.
\end{aligned}
\end{equation*}
Defining $\epsilon_n=u_n-u(t_n)$, we rewrite the previous expression as 
\begin{equation}\label{eq:rec1}
    \epsilon_{n+1} = \rme^{\tau A}\rme^{\tau B}\epsilon_{n}+\delta_{n+1},\quad 
    \delta_{n+1}=\sum_{k=1}^4 \delta_{n+1,k},
\end{equation}
where
\begin{align*}
    \delta_{n+1,1}&=\left(\rme^{\tau A}\rme^{\tau B}-\rme^{\tau L}\right)u(t_n),\\
    \delta_{n+1,2}&=\tau\left(\varphi_1(\tau A)\varphi_1(\tau B)
    -\varphi_1(\tau L)\right)h (t_n),\\
    \delta_{n+1,3}&=\tau\varphi_1(\tau A)\varphi_1(\tau B)(g(u_n)-h(t_n)),\\
    \delta_{n+1,4}&=-\int_0^\tau \rme^{(\tau-s)L}\int_0^s h'(t_n+\sigma)d\sigma ds.
\end{align*}
Solving recurrence \eqref{eq:rec1} with $\epsilon_0=0$, we get
\begin{equation*}
    \epsilon_n 
    = \sum_{j=0}^{n-1}\left(\rme^{\tau A}\rme^{\tau B}\right)^{n-j-1} \delta_{j+1}
    =\sum_{j=0}^{n-1}\left(\rme^{\tau A}\rme^{\tau B}\right)^j \delta_{n-j},
\end{equation*}
and, since Assumption~\ref{assumption:contraction} holds true, we obtain
\begin{equation*}
    \lVert\epsilon_n\rVert 
    \leq \sum_{j=0}^{n-1}\left\lVert\rme^{\tau A}\right\rVert^j
    \left\lVert\rme^{\tau B}\right\rVert^j
    \lVert\delta_{n-j}\rVert
    \leq \sum_{j=0}^{n-1} \lVert\delta_{j+1}\rVert 
    \leq \sum_{j=0}^{n-1}\sum_{k=1}^4 \lVert\delta_{j+1,k}\rVert.
\end{equation*}
Then, under Assumption~\ref{assumption:order1}, we have  
\begin{equation*}
    \lVert\delta_{j+1,1}\rVert\leq C\tau^2,\quad 
    \lVert\delta_{j+1,2}\rVert\leq C\tau^2,\quad
    \lVert\delta_{j+1,3}\rVert\leq C\tau \lVert\epsilon_{j}\rVert,\quad
    \lVert\delta_{j+1,4}\rVert\leq C\tau^2,
\end{equation*}
where, in particular, the first two inequalities follow from 
Proposition~\ref{prop:splitorder}, while the other two follow from
the fact that $g$ is Lipschitz and its
derivative is uniformly bounded, respectively. Coming back to the global
error, the previous bounds lead to
\begin{equation*}
    \lVert\epsilon_n\rVert \leq C\tau+C\tau\sum_{j=1}^{n-1}\lVert
    \epsilon_{j}\rVert
\end{equation*}
and, applying the discrete Gronwall lemma~\cite[Lemma 2.15]{HO10}, we conclude.
\end{proof}
\begin{remark}\label{rk:lstab}
The exponential Euler method \eqref{eq:eeg} can be equivalently written as
\begin{equation*}
    \tilde u_{n+1}=\tilde u_n+\tau\varphi_1(\tau L)(L\tilde u_n+g(\tilde u_n)),
\end{equation*}
and, employing approximation~\eqref{eq:phisplit} in this formulation, we
obtain a split scheme different from ERK1L~\eqref{eq:erk1l}, namely
\begin{equation}\label{eq:seef}
    u_{n+1} = u_n+\tau\varphi_1(\tau A)\varphi_1(\tau B)(Lu_n+g(u_n)).
\end{equation}
However, the latter does not show favourable stability properties. Indeed,
similarly to what has been done in \cite{BM22},
we now study the stability of the methods
for a fixed time step size by considering
the extension of the Dahlquist equation
\begin{equation*}
\left\{
  \begin{aligned}
    u'(t) &= (\alpha+\beta) u(t),
    \quad \alpha,\beta\in\CC, \\
    u(0) &= u_0.
\end{aligned}\right.
\end{equation*}
Concerning ERK1L,
it is straightforward to get that the stability function is
$\rme^{z}$, where $z=\tau(\alpha+\beta)$,
while the one of scheme~\eqref{eq:seef} is
\begin{equation*}
    R(z_\alpha,z_\beta)=
     1+(\rme^{z_\alpha }-1)
    \left(\frac{\rme^{z_\beta}-1}{z_\beta}\right)+
    \left(\frac{\rme^{z_\alpha}-1}{z_\alpha}\right)
    (\rme^{z_\beta }-1),\quad z_\alpha=\tau\alpha,\ z_\beta=\tau\beta.
\end{equation*}
It is easy to check
that $\lvert R(z_\alpha,z_\beta)\rvert\to1$
for $\Re(z_\alpha),\Re(z_\beta)\to-\infty$.
Therefore, for problems
having solution with high frequency components, 
scheme~\eqref{eq:seef} may produce spurious
oscillations in the numerical solution, differently from ERK1L.
This is the case, for instance, of the experiment in 
Subsection~\ref{ssec:michaelismenten}. 
\end{remark}

\subsection{Split exponential integrator ERK2L}
We consider now a split version of the second-order exponential
Runge--Kutta scheme known as ETD2RK \cite[formula (22)]{CM02}
\begin{equation}\label{eq:erk}
    \begin{aligned}
      \tilde U_{n2} &= \rme^{\tau L}\tilde u_n
      +\tau\varphi_1(\tau L)g(\tilde u_n),\\  
      \tilde u_{n+1} &= \tilde U_{n2}+
      \tau\varphi_2(\tau L)(g(\tilde U_{n2})-g(\tilde u_n)). 
    \end{aligned}
\end{equation}
Since we are interested in deriving a second-order split integrator, we look
for a local
error of the scheme proportional to $\tau^3$. From this
consideration, we expect that the $\varphi_1$ and the $\varphi_2$ can be
replaced by the second-order approximations in formula \eqref{eq:phisplit},
since they are multiplied by an additional $\tau$. This is not the case for
$\rme^{\tau L}$, and therefore we should employ a higher-order approximation.
A natural
choice is
\begin{equation}\label{eq:strang}
    \rme^{\tau L} \approx \rme^{\frac{\tau}{2}A}\rme^{\tau B}\rme^{\frac{\tau}{2}A},
\end{equation}
which corresponds to the Strang splitting of the exponential function.
Then, inserting \eqref{eq:phisplit} and 
\eqref{eq:strang} in ETD2RK, we obtain the following split integrator that we
call ERK2L:
\begin{equation}\label{eq:erk2l}
    \begin{aligned}
        U_{n2} &= \rme^{\frac{\tau}{2}A} \rme^{\tau B} \rme^{\frac{\tau}{2}A} u_n
        +\tau\varphi_1(\tau A) \varphi_1(\tau B)g(u_n),\\  
        u_{n+1} &= U_{n2}+2\tau\varphi_2(\tau A)\varphi_2(\tau B)(g(U_{n2})
        -g(u_n)).
    \end{aligned}
\end{equation}
Notice that this integrator coincides with the homonym presented
in \cite[formula~(3.2)]{CCEO26} for commuting operators
$A$ and $B$.

Similarly to what done before, we start by analyzing
Strang splitting~\eqref{eq:strang}, getting an explicit
representation of the error.
\begin{lemma}\label{lemma:strang}
  The following representation of the
  Strang splitting error for \eqref{eq:strang} holds:
\begin{equation}\label{eq:errstrang}
    \begin{aligned}
        &\rme^{\tau L}-\rme^{\frac{\tau}{2}A}\rme^{\tau B}\rme^{\frac{\tau}{2}A}\\
        &=\tau^3 \left[\varphi_3(\tau L)L^3
        -\frac{1}{8}\left(\left(\rme^{\frac{\tau}{2}A}\rme^{\tau B}+I\right)
        \varphi_3\left(\frac{\tau}{2}A\right)
        +\varphi_2\left(\frac{\tau}{2} A\right)
        +\frac{1}{2}\varphi_1\left(\frac{\tau}{2} A\right)\right)A^3 \right.\\
        &\quad\left.   
        -\rme^{\frac{\tau}{2}A}\varphi_3(\tau B)B^3
        -\frac{1}{4}\varphi_2\left(\frac{\tau}{2}A\right)A^2 B
        -\frac{1}{4}\varphi_1\left(\frac{\tau}{2}A\right)A B^2
        -\frac{1}{2}\rme^{\frac{\tau}{2}A}\varphi_2(\tau B)B^2 A
        \right.\\
        &\quad\left.
        -\frac{1}{8}\rme^{\frac{\tau}{2}A}\varphi_1(\tau B)B A^2
        -\frac{1}{4}\varphi_1\left(\frac{\tau}{2}A\right)ABA \right].
    \end{aligned}
\end{equation}
\end{lemma}
\begin{proof}
To prove \eqref{eq:errstrang} we proceed as in Lemma~\ref{lemma:phisplit},
that is we repeatedly apply relation~\eqref{eq:phirel}. First, we get
\begin{equation*}
    \begin{aligned}
        &\rme^{\tau L}-\rme^{\frac{\tau}{2}A}\rme^{\tau B}\rme^{\frac{\tau}{2}A}\\ 
        &=\rme^{\tau L}-I-\left(\rme^{\frac{\tau}{2}A}\rme^{\tau B}+I\right)
        \left(\rme^{\frac{\tau}{2}A}-I\right)-\rme^{\frac{\tau}{2}A}
        \left(\rme^{\tau B}-I\right)\\
        &=\tau\left[\varphi_1(\tau L)L-\frac{1}{2}\left(\rme^{\frac{\tau}{2}A}
        \rme^{\tau B}+I\right)\varphi_1\left(\frac{\tau}{2}A\right)A
        -\rme^{\frac{\tau}{2}A}\varphi_1(\tau B)B\right].\\
    \end{aligned}
\end{equation*}
Then, adopting the same strategy, we rewrite the previous identity as
\begin{equation*}
    \begin{aligned}
        &\tau\left[\left(\varphi_1(\tau L)-I\right)L
        -\frac{1}{2}\left(\left(\rme^{\frac{\tau}{2}A}\rme^{\tau B}+I\right)
        \left(\varphi_1\left(\frac{\tau}{2}A\right)-I\right)+\rme^{\frac{\tau}
        {2}A}-I\right)A \right.\\
        &\quad\left.-\rme^{\frac{\tau}{2}A}\left(\varphi_1(\tau B)-I\right)B
        -\left(\rme^{\frac{\tau}{2}A}-I\right)B-\frac{1}{2}\rme^{\frac{\tau}{2}A}
        \left(\rme^{\tau B}-I\right)A\right]\\
        &=\tau^2\left[\varphi_2(\tau L)L^2
        -\frac{1}{4}\left(\left(\rme^{\frac{\tau}{2}A}\rme^{\tau B}+I\right)
        \varphi_2\left(\frac{\tau}{2}A\right)
        +\varphi_1\left(\frac{\tau}{2}A\right)\right)A^2 \right.\\
        &\quad\left. -\rme^{\frac{\tau}{2}A}\varphi_2(\tau B)B^2
        -\frac{1}{2}\varphi_1\left(\frac{\tau}{2}A\right)AB
        -\frac{1}{2}\rme^{\frac{\tau}{2}A}\varphi_1(\tau B)BA\right]
    \end{aligned}
\end{equation*}
Finally, the last expression is manipulated to obtain 
\begin{equation*}
    \begin{aligned}
        &\tau^2\left[\left(\varphi_2(\tau L)-\frac{1}{2}I\right)L^2
        -\frac{1}{4}\left(\left(\rme^{\frac{\tau}{2}A}\rme^{\tau B}+I\right)
        \left(\varphi_2\left(\frac{\tau}{2}A\right)-\frac{1}{2}I\right)
        +\varphi_1\left(\frac{\tau}{2}A\right)-I \right.\right.\\
        &\left.\left.
        +\frac{1}{2}\left(\rme^{\frac{\tau}{2}A}-I\right)
        +\frac{1}{2}\rme^{\frac{\tau}{2}A}\left(\rme^{\tau B}-I\right)\right)A^2 
        -\left(\rme^{\frac{\tau}{2}A}\left(\varphi_2(\tau B)-\frac{1}{2}I\right)+\frac{1}{2}\left(\rme^{\frac{\tau}{2}A}-I\right)\right)B^2 \right.\\
        &\left.-\frac{1}{2}\left(\varphi_1\left(\frac{\tau}{2}A\right)-I\right)AB
        -\frac{1}{2}\left(\rme^{\frac{\tau}{2}A}\left(\varphi_1(\tau B)-I\right)+\rme^{\frac{\tau}{2}A}-I\right)BA\right]
    \end{aligned}
\end{equation*}
and we conclude employing again \eqref{eq:phirel}.  
\end{proof}
As expected, the total degree of the monomials made by the unbounded
operators $A$ and $B$ in formula~\eqref{eq:errstrang} is three.
Moreover, this representation of the Strang
splitting error differs from those
presented in \cite{JL00,HO09,IK24}.

We now define the set $\mathcal{D}_3$ as the intersection
of the domains of all the monomials made by $A$ and $B$ of total degree
three.
Thanks to the previous result, we are able to obtain the following local error
bound.
\begin{proposition}\label{prop:splitorderstrang}
    Under Assumption~\ref{assumption:semigroup}, if $v\in\mathcal{D}_3$ the
    following bound on the local error of approximation \eqref{eq:strang} 
    holds true:
    \begin{equation*}
        \lVert \rme^{\tau L}v-\rme^{\frac{\tau}{2} A}\rme^{\tau B}
        \rme^{\frac{\tau}{2} A}v\rVert \leq C\tau^3
    \end{equation*}
\end{proposition}
\begin{proof}
The proof is the same as the one in Proposition~\ref{prop:splitorder}, using 
identity \eqref{eq:errstrang}.
\end{proof}

Finally, before stating the convergence result for ERK2L, we consider the
following regularity assumption. 

\begin{assumption}\label{assumption:order2}
Problem \eqref{eq:pde} has solution $u(t)\in\mathcal{D}_3$, and
$g(u(t))\in\mathcal{D}_2$. Moreover, $g$ is twice Fr{\'e}chet 
differentiable, and its derivatives are uniformly bounded.
\end{assumption}

\begin{theorem}\label{thm:serk2}
Let Assumptions \ref{assumption:semigroup}, \ref{assumption:contraction} and 
\ref{assumption:order2} be valid. Then, the following bound on the global error
of scheme ERK2L~\eqref{eq:erk2l} holds:
\begin{equation*}
    \lVert u_n-u(t_n)\rVert\leq C\tau^2,\quad 0\leq n\tau\leq T.
\end{equation*}
The constant $C$ is independent of $\tau$ and $n$, but it depends on $T$.
\end{theorem}

\begin{proof}
  We start the proof by considering the following Taylor expansion of
  $h(t)=g(u(t))$:
\begin{equation}\label{eq:h2}
    h(t_n+s) = h(t_n)+sh'(t_n)+\int_0^s (s-\sigma)h''(t_n+\sigma)d\sigma.
\end{equation}
Then, we insert it in the solution $u(t_{n+1})$ written in terms of
variation-of-constants formula~\eqref{eq:voc}.
Then, we take the difference between the intermediate stage $U_{n2}\approx 
u(t_{n+1})$ and the exact solution, obtaining 
\begin{align*}
    U_{n2}-u(t_{n+1})
    &=\rme^{\frac{\tau}{2}A}\rme^{\tau B}\rme^{\frac{\tau}{2}A}u_n
    -\rme^{\tau L}u(t_n)+\tau\varphi_1(\tau A)\varphi_1(\tau B)g(u_n)
    -\tau\varphi_1(\tau L)h(t_n) \\
    &\quad-\tau^2\varphi_2(\tau L)h'(t_n)
    -\int_0^\tau \rme^{(\tau-s)L}\int_0^s (s-\sigma)h''(t_n+\sigma)d\sigma ds.
\end{align*}
By defining $\epsilon_n = u_n-u(t_n)$, the last expression can be rewritten as
\begin{equation}\label{eq:diffn2}
    U_{n2}-u(t_{n+1})=\rme^{\frac{\tau}{2}A}\rme^{\tau B}\rme^{\frac{\tau}{2}A}
    \epsilon_n+\sum_{k=1}^4 \delta_{n+1,k}-\tau^2\varphi_2(\tau L)h'(t_n),
\end{equation}
where
\begin{align*}
    \delta_{n+1,1} &= \left(\rme^{\frac{\tau}{2}A}\rme^{\tau B}\rme^{\frac{\tau}{2}A}
    -\rme^{\tau L}\right)u(t_n), \\
    \delta_{n+1,2}&=\tau\varphi_1(\tau A)\varphi_1(\tau B)(g(u_n)-h(t_{n})), \\
    \delta_{n+1,3}&=\tau(\varphi_1(\tau A)\varphi_1(\tau B)
    -\varphi_1(\tau L))h(t_n), \\
    \delta_{n+1,4} &= -\int_0^\tau \rme^{(\tau-s)L}\int_0^s (s-\sigma)
    h''(t_n+\sigma)d\sigma ds.
\end{align*}
The last term in the right-hand side of formula~\eqref{eq:diffn2}
can be rewritten thanks to the identity
\begin{equation*}
-\tau h'(t_n)=
  -(h(t_{n+1})-h(t_n))
    +\int_0^\tau (\tau-\sigma)h''(t_n+\sigma) d\sigma,
\end{equation*}
which follows from \eqref{eq:h2} taking $s=\tau$.
In addition, we consider $u_{n+1}-u(t_{n+1})$, and insert the definition of
$u_{n+1}$, thus obtaining
\begin{equation*}
    \epsilon_{n+1}= U_{n2}-u(t_{n+1})+2\tau\varphi_2(\tau A)
    \varphi_2(\tau B)(g(U_{n2})-g(u_n)),
\end{equation*}
that is
\begin{equation}\label{eq:rec2}
\epsilon_{n+1}=
  \rme^{\frac{\tau}{2}A}\rme^{\tau B}\rme^{\frac{\tau}{2}A}\epsilon_n
+\delta_{n+1},\quad
\delta_{n+1}=\sum_{k=1}^8\delta_{n+1,k},
\end{equation}
where
\begin{align*}
    \delta_{n+1,5}&=2\tau\varphi_2(\tau A)\varphi_2(\tau B)(g(U_{n2})
    -h(t_{n+1})),\\
    \delta_{n+1,6} &= 2\tau\varphi_2(\tau A)\varphi_2(\tau B)(h(t_n)-g(u_n)),\\
    \delta_{n+1,7} &= \tau(2\varphi_2(\tau A)\varphi_2(\tau B)-\varphi_2(\tau L))
    (h(t_{n+1})-h(t_n)),\\
    \delta_{n+1,8} &= \tau\varphi_2(\tau L)\int_{0}^\tau (\tau-\sigma)
    h''(t_n+\sigma)d\sigma.
\end{align*}
Then, solving the 
recursion~\eqref{eq:rec2} with $\epsilon_0=0$, we get
\begin{equation*}
    \epsilon_n 
    = \sum_{j=0}^{n-1} \left(\rme^{\frac{\tau}{2}A}\rme^{\tau B}
    \rme^{\frac{\tau}{2}A}\right)^{n-j-1} \delta_{j+1}
    = \sum_{j=0}^{n-1} \left(\rme^{\frac{\tau}{2}A}\rme^{\tau B}
    \rme^{\frac{\tau}{2}A}\right)^j \delta_{n-j}
\end{equation*}
and, under Assumption \ref{assumption:contraction}, it follows that
\begin{equation*}
  \lVert\epsilon_n\rVert \leq \sum_{j=0}^{n-1}\left\lVert\rme^{\frac{\tau}{2} A}
  \right\rVert^j\left\lVert
    \rme^{\tau B}\right\rVert^j\left\lVert\rme^{\frac{\tau}{2} A}\right\rVert^j \lVert\delta_{n-j}\rVert 
    \leq \sum_{j=0}^{n-1}\sum_{k=1}^8 \lVert\delta_{j+1,k}\rVert.
\end{equation*}
From Assumption \ref{assumption:order2} and Proposition 
\ref{prop:splitorderstrang},
we have that $\lVert\delta_{j+1,1}\rVert$, $\lVert\delta_{j+1,3}\rVert$ and 
$\lVert\delta_{j+1,7}\rVert$ are bounded by $C\tau^3$, since they are 
splitting errors. Moreover, it follows from the boundedness of the
derivatives of
$g$ that $\lVert\delta_{j+1,4}\rVert$ and $\lVert\delta_{j+1,8}\rVert$
are bounded by $C\tau^3$. Finally, since $g$ is
Lipschitz, $\lVert\delta_{j+1,2}\rVert$ and $\lVert\delta_{j+1,6}\rVert$
are bounded by 
$C\tau\lVert\epsilon_{j}\rVert$, and, using in addition \eqref{eq:diffn2},
we get
$\lVert\delta_{j+1,5}\rVert\le C\tau\lVert\epsilon_j\rVert+C\tau^3$.
Thanks to these inequalities, we can apply a
discrete Gronwall lemma to
\begin{equation*}
    \lVert\epsilon_n\rVert\leq C\tau^2+C\tau\sum_{j=1}^{n-1}\lVert
    \epsilon_{j}\rVert,
\end{equation*}
and conclude.
\end{proof}

\section{Numerical examples}\label{sec:examples}

In this section, we apply the split integrators ERK1L~\eqref{eq:erk1l}
and ERK2L~\eqref{eq:erk2l} to solve
two PDEs containing noncommuting operators. In the first example, we validate 
the theoretical results by applying the two
schemes to a dispersive problem with a manufactured exact solution,
employing a truncated Fourier expansion for the implementation.
In the second example, we solve a physically relevant
diffusion--reaction equation, demonstrating
the efficiency of the approach compared to the unsplit counterpart.
In this case, the PDE is numerically solved by
applying the so-called method of lines with second-order finite
differences, converting the unbounded spatial
differential operators into matrices with arbitrary large norm.
Then, the proposed time-marching schemes are applied to the arising system
of ODEs.
\subsection{Nonlinear dispersive equation}
\label{ssec:drncc}
\begin{figure}[ht]
\centering
%
%
\definecolor{mycolor1}{rgb}{0.00000,0.44706,0.69804}%
\definecolor{mycolor2}{rgb}{0.90196,0.62353,0.00000}%
\definecolor{mycolor3}{rgb}{0.12941,0.12941,0.12941}%
\begin{tikzpicture}

\begin{axis}[%
width=1.535in,
height=1.535in,
at={(0.706in,0.632in)},
scale only axis,
xmode=log,
xmin=0.001,
xmax=0.1,
xminorticks=true,
xlabel style={font=\color{mycolor3}},
xlabel={$\tau$},
ymode=log,
ymin=1e-06,
ymax=0.1,
yminorticks=true,
ylabel style={font=\color{mycolor3}},
ylabel={Relative error},
axis background/.style={fill=white},
legend style={at={(0.97,0.03)}, font=\footnotesize, anchor=south east, legend
cell
align=left, align=left}
]
\addplot [color=mycolor1, line width=1.2pt, only marks, mark size=3.0pt, mark=o,
mark options={solid, mycolor1}]
  table[row sep=crcr]{%
0.03125	0.0327529856858712\\
0.0178571428571429	0.0188116440183811\\
0.01	0.0105671674640103\\
0.00561797752808989	0.00594703786080898\\
0.00316455696202532	0.00335323446580755\\
};
\addlegendentry{ERK1L}

\addplot [color=mycolor2, line width=1.2pt, only marks, mark size=3.0pt,
mark=square, mark options={solid, mycolor2}]
  table[row sep=crcr]{%
0.03125	0.000753811106569102\\
0.0178571428571429	0.000245383654163059\\
0.01	7.68131556956047e-05\\
0.00561797752808989	2.42190945424106e-05\\
0.00316455696202532	7.68031512036317e-06\\
};
\addlegendentry{ERK2L}

\addplot [color=black, densely dotted, line width=1.2pt, forget plot]
  table[row sep=crcr]{%
0.03125	0.0331131903498496\\
0.0178571428571429	0.0189218230570569\\
0.01	0.0105962209119519\\
0.00561797752808989	0.00595293309660217\\
0.00316455696202532	0.00335323446580755\\
};
\addplot [color=black, densely dashdotted, line width=1.2pt, forget plot]
  table[row sep=crcr]{%
0.03125	0.000748950729159165\\
0.0178571428571429	0.000244555340133605\\
0.01	7.66925546658985e-05\\
0.00561797752808989	2.42054521733047e-05\\
0.00316455696202532	7.68031512036317e-06\\
};
\end{axis}

\end{tikzpicture}%
\caption{Relative errors in the
  $L^2$-norm of the numerical solutions of 
\eqref{eq:experiodic} obtained applying the schemes ERK1L and ERK2L with 
different time steps $\tau$, ranging from $1/316$ to $1/32$. 
The spatial discretization is done with 300 Fourier modes for each
direction,
and the final time of the simulation is $T=1$. 
The dotted reference line has slope $1$, while the dash-dotted one has slope
$2$.} 
\label{fig:ex_periodic_order}
\end{figure}
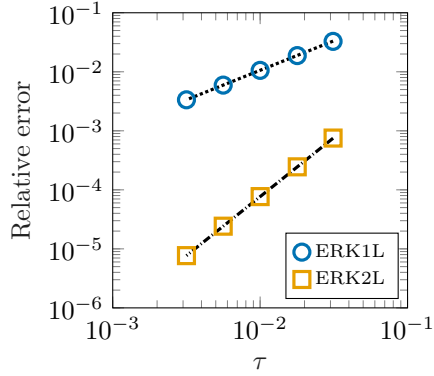
We consider the PDE 
\begin{equation}\label{eq:experiodic}
    \partial_t u = \rmi(\partial_{xx}u+b(x)\partial_{yy}u)+g(t,x,y,u)
    \quad \text{on $\Omega$},
\end{equation}
coupled with periodic boundary conditions,
where $\Omega=(0,2\pi)^2$ and $b(x)=2+\sin(x)$ is a spatially variable
coefficient. 
The nonlinear reaction term is
\begin{equation*}
    g(t,x,y,u)=\frac{1}{1+u^2}+s(t,x,y),
\end{equation*}
where the source term $s(t,x,y)$ and the initial condition
are such that the exact solution of the problem is 
\begin{equation*}
    u(t,x,y) = \rme^{\rmi t} \sin(x)\cos(y).
\end{equation*}
We apply the schemes ERK1L and ERK2L to solve the problem up to the final time
$T=1$, choosing the operators $A=\rmi\partial_{xx}$ and 
$B=\rmi b(x)\partial_{yy}$ to split the spatial differential term
in~\eqref{eq:experiodic}.
The assumptions in Theorems~\ref{thm:serk1} and \ref{thm:serk2} are satisfied
with $X=L^2(\Omega)$ and the $L^2$-norm. In fact, 
given the Fourier series of a function $v$
\begin{equation*}
    v(x,y) = \sum_{j,k=-\infty}^{+\infty} \hat{v}_{jk}\rme^{\rmi(jx+ky)}
\end{equation*}
we can define
\begin{equation*}
    (\rme^{tA}v)(x,y)= 
    \sum_{j,k=-\infty}^{+\infty} \rme^{-t\rmi j^2}\hat{v}_{jk}
    \rme^{\rmi(jx+ky)},
    \quad     (\rme^{tB}v)(x,y) 
    = \sum_{j,k=-\infty}^{+\infty}\rme^{-t\rmi k^2b(x)}
    \hat{v}_{jk} \rme^{\rmi (jx+ky)}.
\end{equation*}
Clearly, these operators have the properties of a strongly
continuous semigroup. Moreover,
$\lVert\rme^{tA}v\rVert^2=
\lVert v\rVert^2$, hence $\lVert \rme^{tA}\rVert=1$ and similarly
for $\lVert \rme^{tB}\rVert$. So,
Assumptions~\ref{assumption:semigroup} and \ref{assumption:contraction}
are satisfied. Finally, $u$ and $g$ are smooth periodic functions, hence
Assumptions~\ref{assumption:order1} and \ref{assumption:order2} are
also satisfied. Therefore, we expect the two schemes
to converge with order one and two, respectively. The
numerical experiments confirm this behavior, as illustrated 
in Figure~\ref{fig:ex_periodic_order}. Concerning the actual implementation,
we truncated the above Fourier series to 300 modes for each direction, and
employed Fast Fourier Transform techniques.
\subsection{Enzyme kinetics of Michaelis--Menten type}
\label{ssec:michaelismenten}
In this second example, we consider the enzyme kinetics of Michaelis--Menten
type diffusion--reaction equation \cite{AAKW20,BK15}.
The unknown represents the
concentration of the substrate, while the nonlinear term describes the reaction
between the substrate and the enzyme. During the dynamics, the substrate is
progressively consumed and its concentration approaches the zero
stationary state. We choose a sector of a unit disk with a central 
angle of $\pi/3$ as the spatial domain of the problem.
The initial condition is set to the constant value one,
while homogeneous Dirichlet boundary conditions are imposed.
Given the geometry of the problem,
it appears natural to employ polar coordinates,
so that the spatial variables are
$(\rho,\theta)\in\Omega=(0,1)\times(-\pi/6,\pi/6)$
and the PDE is
\begin{equation}\label{eq:mm}
  \partial_t u = d\left(
\frac{1}{\rho}\frac{\partial}{\partial\rho}
    +\frac{\partial^2}{\partial \rho^2}
    +\frac{1}{\rho^2}\frac{\partial^2}{\partial\theta^2}\right)
  u -\frac{u}{1+u} \quad \text{on $\Omega$},
\end{equation}
where $d=0.2$ is the diffusion coefficient.
Numerical methods lacking favorable stability properties may produce
solutions of \eqref{eq:mm} that exhibit oscillations
(see \cite[Fig.~15]{BK15}),
due to the discontinuity between the
initial and boundary conditions.
We apply the method of lines and
use standard second-order finite differences for the spatial 
discretization. We denote by $n_\rho$ and $n_\theta$ the number of inner
points used to uniformly discretize the variables
$\rho$ and $\theta$, respectively, so that the corresponding
constant step sizes are $h_\rho=1/(n_\rho+1)$ and $h_\theta=\pi/(3
(n_\theta+1))$. The resulting matrices of size $N\times N$, with
$N=n_\rho n_\theta$, are
\begin{equation*}
  A=I_\theta \otimes M_\rho\approx
  \frac{1}{\rho}\frac{\partial}{\partial\rho}+\frac{\partial^2}
    {\partial \rho^2},\quad
    B=    M_\theta \otimes D_\rho  \approx
    \frac{1}{\rho^2}\frac{\partial^2}{\partial\theta^2},
\end{equation*} 
where $I_\theta$ is the identity matrix of size $n_\theta\times n_\theta$ and
$D_\rho$ is the 
diagonal matrix whose nonzero entries are $1/\rho_i^2$, where
$\rho_i=i h_\rho$, 
$i=1,\dots,n_\rho$. Moreover, $M_\theta$ is the classical symmetric tridiagonal matrix
of size $n_\theta\times n_\theta$ with stencil $[1,-2,1]/h_\theta^2$, while
$M_\rho$ is the tridiagonal matrix of size $n_\rho\times n_\rho$ with
constant diagonal elements $-2/h_\rho^2$, and
superdiagonal elements $b_i$ and
subdiagonal elements $c_i$ defined by
\begin{equation*}
  b_i=\left(1+\frac{1}{2i}\right)\frac{1}{h_\rho^2},\quad
  c_i=\left(1-\frac{1}{2(i+1)}\right)\frac{1}{h_\rho^2},\quad
  i=1,\ldots,n_\rho-1.
\end{equation*}
Finally, the symbol
$\otimes$ denotes the standard Kronecker product.

We notice that
$A$ and $B$ have nonnegative
extradiagonal elements, and hence
the matrix exponentials $\rme^{tA}$ and $\rme^{tB}$
(and also the matrices $\varphi_k(tA)$ and $\varphi_k(tB)$,
see~\eqref{eq:phifun}) are
nonnegative for all $t\ge 0$ \cite[Theorem I.7.2]{HV03}. Moreover,
we have
    \begin{equation*}
        \lVert \rme^{t A}\rVert_\infty \leq 1,\quad 
        \lVert \rme^{t B}\rVert_\infty 
        \leq 1,\quad t\geq 0.
    \end{equation*} 
    These bounds can be proved considering the logarithmic $\infty$-norm 
    \cite[Section I.2.3]{HV03}.
    Indeed,
    the sum of the elements of the first row of $M_\rho$ is negative,
    that of the last row
    is negative, while that of each remaining row is zero. Therefore,
        $\mu_\infty(I_\theta\otimes M_\rho)=\mu_\infty(M_\rho) = 0$, and hence
    \begin{equation*}
      \lVert \rme^{t A}\rVert_\infty
\leq         \rme^{t \mu_\infty(A)} = 1.
    \end{equation*}
    Concerning the matrix $B$, with arguments similar to the previous case,
    we again conclude that 
        $\mu_\infty(B)=0$,
    and hence $\lVert \rme^{tB}\rVert_\infty\le 1$.

    Now, we show the computational advantage of computing the action
    of matrix functions of
    $A$ and $B$ \emph{separately} as in the proposed split schemes.
    In fact, for the matrix $A$ and a generic vector $v\in\RR^{N}$    we have
    \begin{equation*}
      \varphi_\ell(\tau A)v=
      (I_\theta\otimes \varphi_\ell(\tau M_\rho))v=
      \mathrm{vec}(\varphi_\ell(\tau M_\rho)V),
    \end{equation*}
    where $\mathrm{vec}$ is the operator which stacks the columns of the input
    matrix into a vector, and $V\in\RR^{n_\rho\times n_\theta}$ is
    such that $v=\mathrm{vec}(V)$
    (see, for instance, \cite{VL00}). The computation of the ``small'' matrix
    $\varphi_\ell(\tau M_\rho)$ can be performed by diagonalization or
    Pad\'e or Taylor  approximations with modified scaling and squaring
    (see, for instance, \cite{H08}).
    Regarding the matrix $B$, we first diagonalize the small
    matrix $M_\theta$ to obtain
    $Q_\theta\Lambda_\theta Q_\theta^{\sf T}$. Then, by using the properties of
    the     Kronecker product, we get
    \begin{equation*}
      \varphi_\ell(\tau B)=
      \varphi_\ell(\tau(Q_\theta\otimes I_\rho)(\Lambda_\theta\otimes D_\rho)
      (Q_\theta^{\sf T}\otimes I_\rho))=
      (Q_\theta\otimes I_\rho)\varphi_\ell(\tau(\Lambda_\theta\otimes D_\rho))
      (Q_\theta^{\sf T}\otimes I_\rho),
    \end{equation*}
    where $I_\rho$ is the identity matrix of size $n_\rho \times n_\rho$.
    Therefore, we can compute the action on the vector $v$ as
    \begin{equation*}
      \varphi_\ell(\tau B)v=
      \mathrm{vec}((\Phi_{\rho\theta}\circ(VQ_\theta))Q_\theta^{\sf T}),
    \end{equation*}
    where $\Phi_{\rho \theta}$ is the matrix with element
    $\varphi_\ell(\tau\lambda_j/\rho_i^2)$
    in position $(i,j)$, being $\lambda_j$ the $j$th element
    of the diagonal matrix $\Lambda_\theta$, while the symbol $\circ$ denotes
    the standard Hadamard product. The advantage of such direct matrix-oriented
    (or tensor-oriented, for higher dimensional problems)
    approaches, which only involves the computation of small matrix
    functions, scalar quantities, and level 3 BLAS,
    has already been demonstrated in many
    instances (see, for instance, \cite{BS17,DASS20,CCEOZ22,CCZ23bis,CC27}).

\begin{figure}[ht]
\centering
  \includegraphics[scale=0.23]{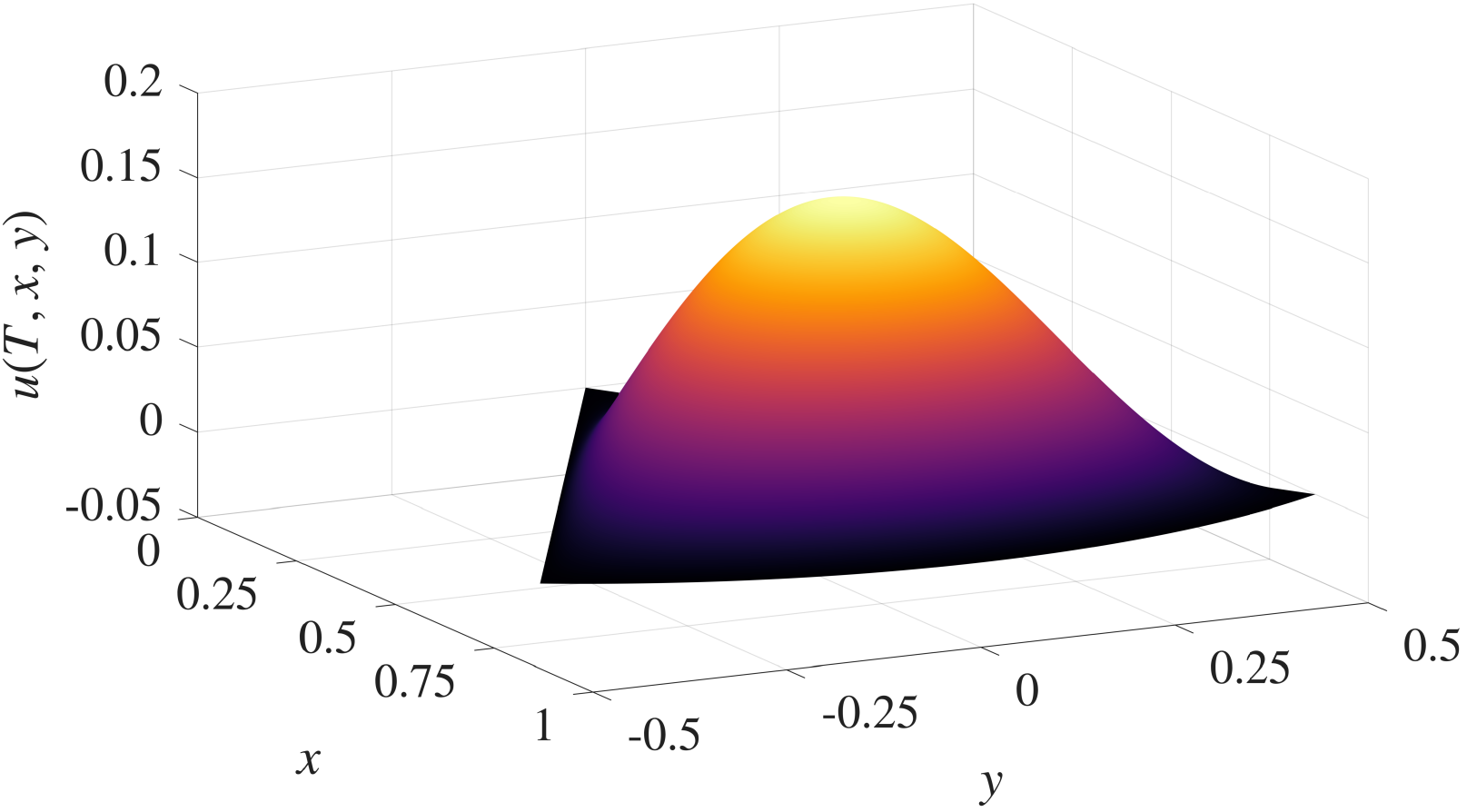}
  \includegraphics[scale=0.23]{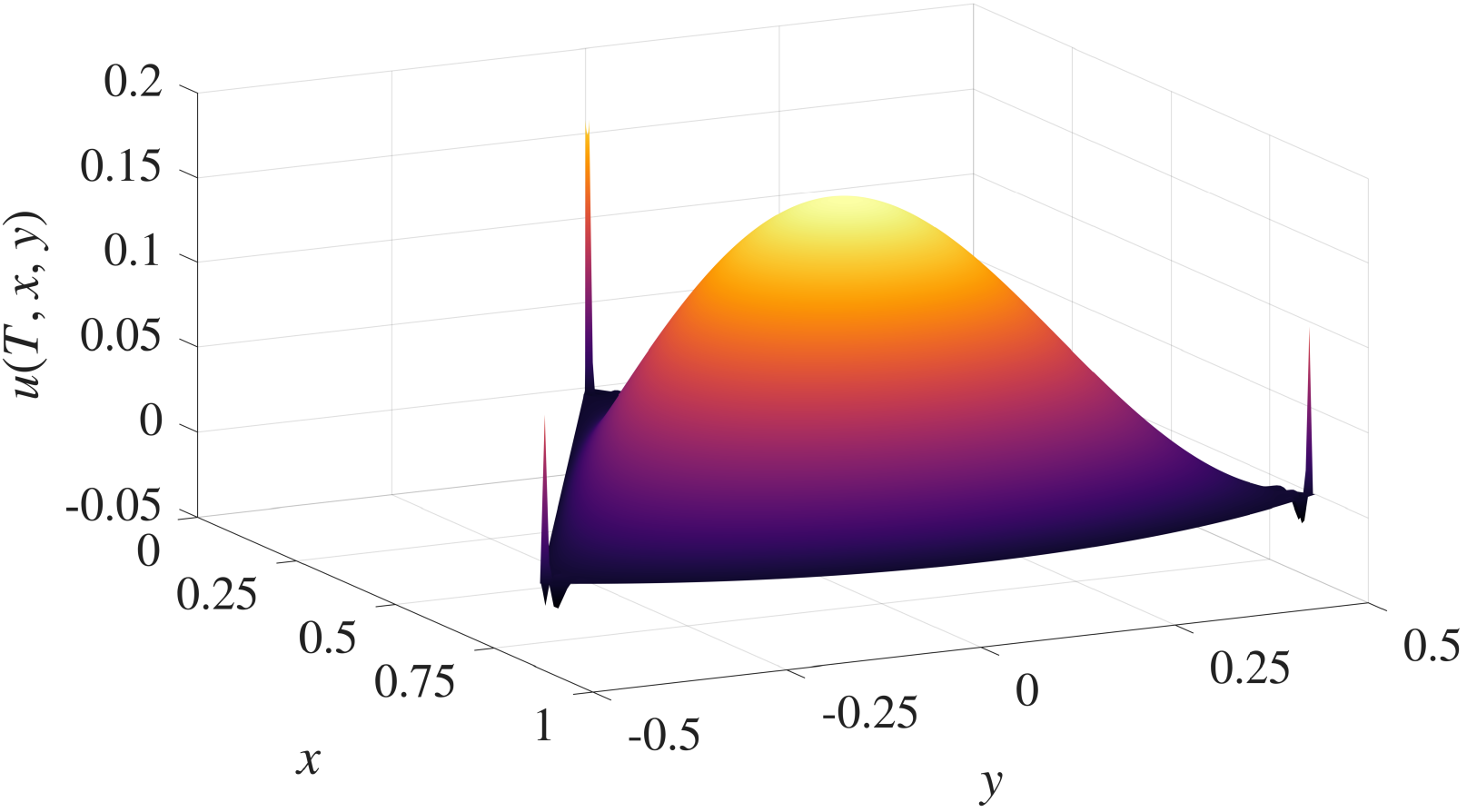}
\caption{Numerical solutions of \eqref{eq:mm} at $T=0.25$ with 
  $(n_\rho,n_\theta)=(200,200)$,
  obtained using ERK1L~\eqref{eq:erk1l} (left) and
  method~\eqref{eq:seef} (right) both with 40 time steps.} 
\label{fig:ex_mm_lstab}
\end{figure}

In the first numerical experiment, we solve \eqref{eq:mm} applying the two
different formulations of the split exponential Euler method described in
Section~\ref{ssec:serk1}, namely ERK1L~\eqref{eq:erk1l} and
method~\eqref{eq:seef}.
The final time of the simulation is $T=0.25$. The
number of points used for the spatial discretization in the two directions is 
$(n_\rho,n_\theta)=(200,200)$, and the number of time steps for both methods is 
40. The numerical solutions are presented in Figure~\ref{fig:ex_mm_lstab}. The
solution produced by method~\eqref{eq:seef} presents spurious
oscillations close to the domain corners, while ERK1L is able to damp
them as expected (see the discussion in Remark~\ref{rk:lstab}).

\begin{figure}[ht]
\centering
\includegraphics[height=3.8cm]{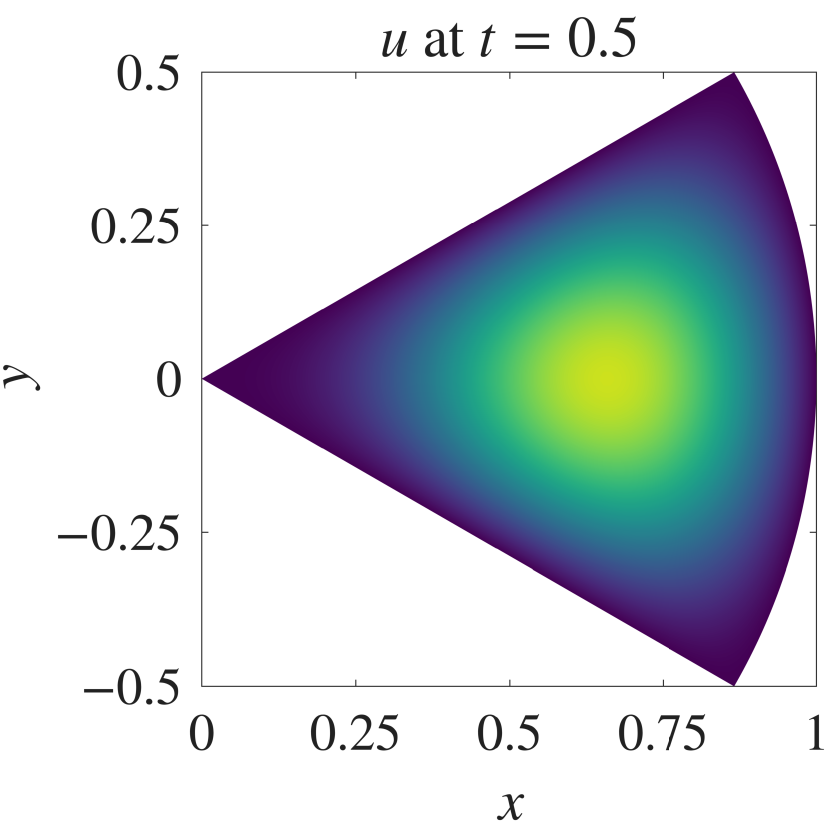}\hfill
\includegraphics[height=3.8cm]{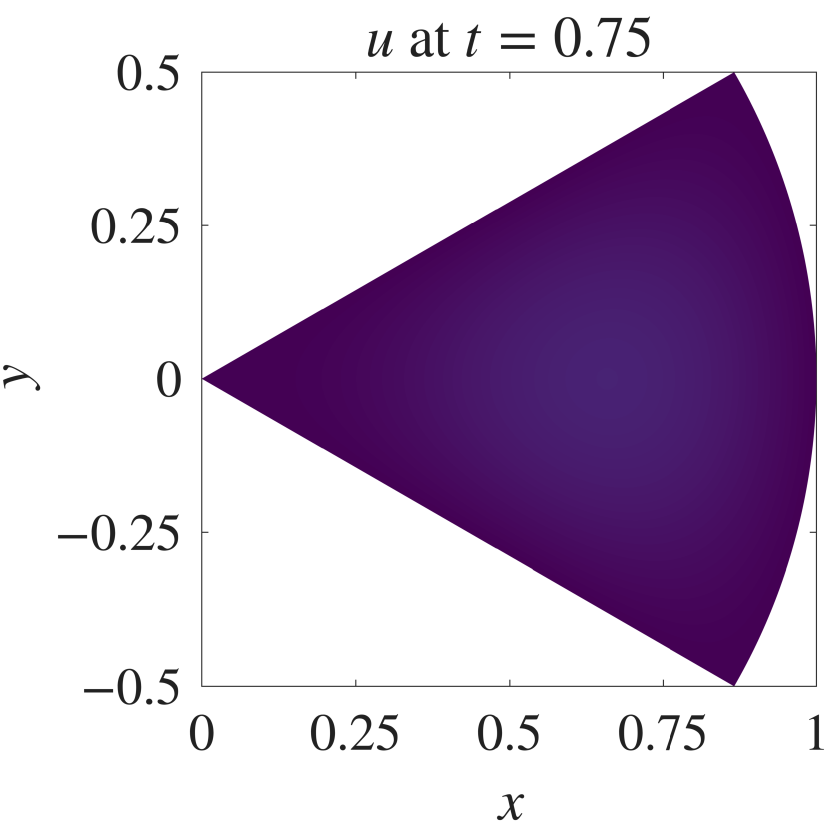}\hfill
\includegraphics[height=3.8cm]{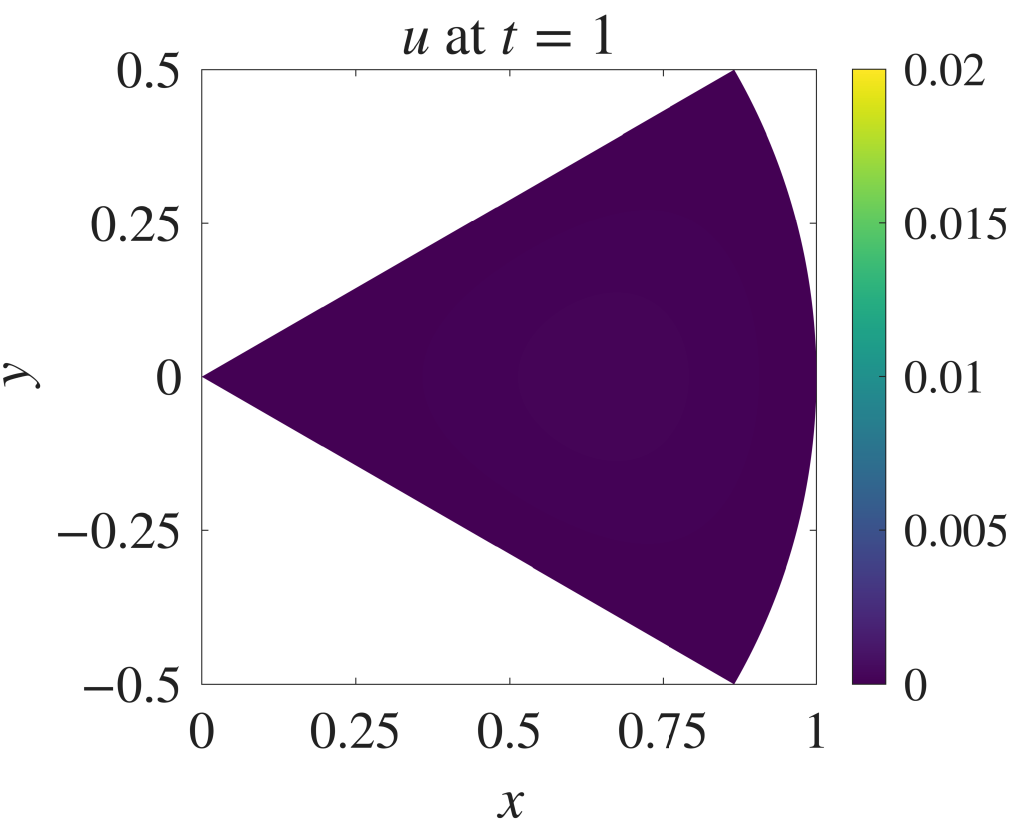}
    \caption{Numerical solution of \eqref{eq:mm} at three time instants with 
      $(n_\rho,n_\theta)=(100,100)$, obtained using ERK2L with time step size
      $\tau=1/3200$.}
    \label{fig:ex_mm_sol}
\end{figure}

In the second experiment, we compare the performances of the split exponential
integrators ERK1L~\eqref{eq:erk1l} and
ERK2L~\eqref{eq:erk2l}
with their unsplit counterparts ETD1RK~\eqref{eq:eeg} and 
ETD2RK~\eqref{eq:erk},
respectively. The linear combinations of actions of $\varphi_\ell(\tau L)$
matrix functions required by the unsplit methods,
where $L=A+B\in\RR^{N\times N}$ is large and sparse,
 are computed
employing the routine KIOPS \cite{GRT18}. This algorithm
is a state-of-the-art Krylov subspace iterative technique with incomplete
orthogonalization. The input tolerance required by KIOPS is set to $10^{-6}$
and 
$10^{-7}$ for ETD1RK and ETD2RK, respectively.
Notice that other Krylov-based relevant techniques to efficiently
compute actions of
exponential-like matrix functions are available in the literature
(see, for instance, \cite{LPR19,CCZ23}). The final time of the simulations
is now set to $T=1$ and for the spatial discretization we
take $(n_\rho,n_\theta)=(100,100)$.
In Figure~\ref{fig:ex_mm_sol}, we plot the solution $u$ 
obtained using ERK2L at three time instants
and we see that, as anticipated, it approaches the zero
stationary state.

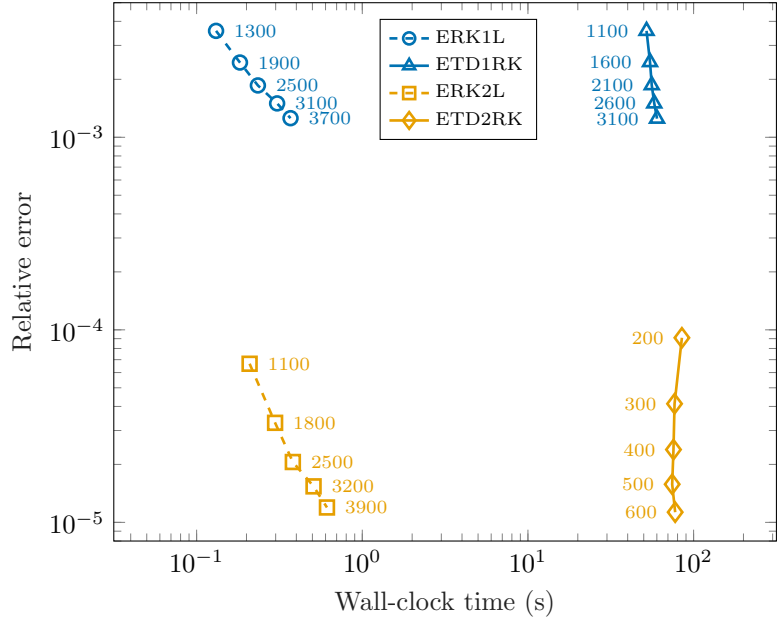
\begin{figure}[ht]
\centering
%
%
\definecolor{mycolor1}{rgb}{0.00000,0.44706,0.69804}%
\definecolor{mycolor2}{rgb}{0.90196,0.62353,0.00000}%
\definecolor{mycolor3}{rgb}{0.12941,0.12941,0.12941}%
\begin{tikzpicture}

\begin{axis}[%
width=3.45in,
height=2.8in,
at={(0.81in,0.722in)},
scale only axis,
xmode=log,
xmin=0.0316227766016838,
xmax=316.227766016838,
xminorticks=true,
xlabel style={font=\color{mycolor3}},
xlabel={Wall-clock time (s)},
ymode=log,
ymin=8e-06,
ymax=0.005,
yminorticks=true,
ylabel style={font=\color{mycolor3}},
ylabel={Relative error},
axis background/.style={fill=white},
legend style={at={(0.52,0.97)}, font=\footnotesize, anchor=north, legend cell
align=left, align=left},
legend image post style={scale=0.87}
]

\addplot [color=mycolor1, dashed, line width=1.0pt, mark size=2.6pt, mark=o,
mark options={solid, mycolor1},
nodes near coords, every node near coord/.append style={
text=mycolor1, font=\footnotesize, anchor=west, xshift=3pt}]
  table[row sep=crcr,
  x index=0, y index=1, meta index=2, point meta=explicit symbolic]{%
0.131032	0.00357092411650545  1300\\
0.182649	0.00244521471420156  1900\\
0.234356	0.0018591556981143  2500\\
0.306175	0.00149972311568235  3100\\
0.368974  0.00125676165296173 3700\\
};
\addlegendentry{ERK1L}

\addplot [color=mycolor1, line width=1.0pt, mark size=3.0pt, mark=triangle, mark
options={solid, mycolor1},
nodes near coords, every node near coord/.append style={
text=mycolor1, font=\footnotesize, anchor=east, xshift=-3pt}]
  table[row sep=crcr,
  x index=0, y index=1, meta index=2, point meta=explicit symbolic]{%
51.989902	0.00357122380930036 1100\\
54.566743	0.0024580607963506 1600\\
55.89723	0.00186543431529149 2100\\
57.882044	0.00150000839442413 2600\\
60.163685 0.00125000993012023 3100\\
};
\addlegendentry{ETD1RK}

\addplot [color=mycolor2, dashed, line width=1.0pt, mark size=2.6pt,
mark=square, mark options={solid, mycolor2},
nodes near coords, every node near coord/.append style={
text=mycolor2, font=\footnotesize, anchor=west, xshift=3pt}]
  table[row sep=crcr,
  x index=0, y index=1, meta index=2, point meta=explicit symbolic]{%
0.208857	6.66579179120926e-05  1100\\
0.297708	3.28908660394087e-05  1800\\
0.381047	2.05831257843591e-05  2500\\
0.50801	1.53823447228052e-05  3200\\
0.612119  1.19504748573839e-05 3900\\
};
\addlegendentry{ERK2L}

\addplot [color=mycolor2, line width=1.0pt, mark size=3.6pt, mark=diamond, mark
options={solid, mycolor2},
nodes near coords, every node near coord/.append style={
text=mycolor2, font=\footnotesize, anchor=east, xshift=-3pt}]
  table[row sep=crcr,
  x index=0, y index=1, meta index=2, point meta=explicit symbolic]{%
84.989727	9.09343233028258e-05  200\\
76.750751	4.1307038259738e-05  300\\
75.611367	2.38807342946966e-05  400\\
74.264829	1.57759250341315e-05  500\\
77.430093 1.1295559638611e-05 600\\
};
\addlegendentry{ETD2RK}

\end{axis}

\end{tikzpicture}%
\caption{Wall-clock times and relative errors in the infinity norm for 
problem~\eqref{eq:mm} at $T=1$ with $(n_\rho,n_\theta)=(100,100)$ using the
first-order methods ERK1L, ETD1RK and the second-order ones ERK2L, ETD2RK. 
The number of time steps employed in each simulation is
reported in the figure.} 
\label{fig:ex_mm_comparison}
\end{figure}

Then, we
compare the methods in terms of  wall-clock time needed to attain different
levels of accuracy.
The experiments are performed on a laptop equipped with an 
Intel\textsuperscript{\textregistered}\ Core\texttrademark\ i7-12700H CPU 
(14 physical cores) and 16GB of RAM using the software MathWorks 
\textsc{Matlab}\textsuperscript{\textregistered}\ 2025b.
The results are presented in Figure~\ref{fig:ex_mm_comparison}.
Despite the higher number of time steps required by the split schemes
to balance the introduced splitting error, the matrix-oriented
approach allows to outperform the unsplit methods in all cases, with a
speedup of about two orders of magnitude.

\section{Conclusions}\label{sec:conclusions}
In this paper, we have investigated the theoretical and practical
advantages of splitting exponential and exponential-like functions
in the context of exponential integrators for semilinear
evolution equations. For two noncommuting unbounded operators,
we established new elementary error representations that
include the classical Lie--Trotter and Strang
splittings. Operating within the abstract framework of
strongly continuous semigroups, we provided rigorous
convergence proofs for split variants of both the first-order
exponential Euler method and a second-order exponential Runge--Kutta scheme.
Numerical experiments on nonlinear dispersive and diffusion--reaction
equations fully validated our theoretical results, showing
significant computational savings compared to
standard unsplit state-of-the-art methods.
Natural directions for future research include extending this semigroup error
analysis to higher-order exponential Runge--Kutta methods and
multi-term operator splittings.

\backmatter
\bmhead{Acknowledgements}
The authors are members of the Gruppo Nazionale Calcolo Scientifico-Istituto
Nazionale di Alta Matematica (GNCS-INdAM).
\bibliography{ncsplit.bib}

\end{document}